\documentclass[12pt]{amsart}
\usepackage{geometry}
\usepackage{graphicx,psfrag,epsfig, color,float}
\usepackage{amssymb,amsmath,amscd,amsthm}
\usepackage{graphicx,psfrag,epsfig}

\numberwithin{equation}{section}
\usepackage{color}
\usepackage{amssymb}
\usepackage{enumerate, xspace}
\usepackage{marvosym} 
\usepackage[sans]{dsfont}        
\usepackage{amsthm}
\usepackage{changebar}
\usepackage[backgroundcolor=blue!20!white, linecolor=blue!20 white,textsize=footnotesize]{todonotes}
\usepackage[colorlinks]{hyperref}
\usepackage{graphicx}
\usepackage{cleveref}
\usepackage{xcolor}
\usepackage{setspace}
\usepackage{mathtools}
\usepackage{verbatim}
\usepackage{enumitem}

\newtheorem{theorem}{Theorem}[section]

\newtheorem{remark}[theorem]{Remark}
\newtheorem{proposition}[theorem]{Proposition}
\newtheorem{definition}[theorem]{Definition}

\newcommand{\be}{\begin{equation}}
\newcommand{\ee}{\end{equation}}
\newcommand{\bi}{\begin{itemize}}
\newcommand{\ei}{\end{itemize}}
\newcommand{\br}{\begin{eqnarray}}
\newcommand{\er}{\end{eqnarray}}

\newcommand{\eps}{\varepsilon}

\newcommand{\commentout}[1]{}

\def\wt{\widetilde}

\def\be{\color{black}}
\def\br{\color{red}}

\def\eps{{\varepsilon}}

\def\fF{\mathfrak{F}}
\def\cF{\mathcal{F}}

\def\rh{\varrho}

\def\EXP{\mathbb{E}}

\def\PROB{\mathbb{P}}
\def\TOR{\mathbb{T}}
\def\real{\mathbb{R}}
\def \bbZ{\mathbb{Z}}

\def\naturals{\mathbb{N}}
\def\complex{\mathbb{C}}

\def\cL{\mathcal{L}}

\def\cM{\mathcal{M}}

\def\Ban{\mathcal{B}}

\def\rh{\varrho}

\date{}

\begin{document}

	\title[Asymptotics for Large Deviations]{Higher order asymptotics for large deviations -- Part II}
	\author{Kasun Fernando, Pratima Hebbar}
	\address{Pratima Hebbar\\
Department of Mathematics\\
University of Maryland \\
4176 Campus Drive\\
College Park, MD 20742-4015, United States.}
\email{{\tt phebbar@math.umd.edu}}

\address{Kasun Fernando\\
Department of Mathematics\\
University of Maryland \\
4176 Campus Drive\\
College Park, MD 20742-4015, United States.}
\email{{\tt abkf@math.umd.edu}}
	
		\begin{abstract}
 We obtain asymptotic expansions for the large deviation principle (LDP) for continuous time stochastic processes with weakly dependent increments. As a key example, we show that additive functionals of solutions of stochastic differential equations (SDEs) satisfying H\"ormander condition on a $d$--dimensional compact manifold admit these asymptotic expansions of all orders.
	\end{abstract}
	

	\maketitle

\section{ Introduction}

Suppose $\{X_n\}_{n\geq 1}$ is a sequence of centred random variables and $S_n = \sum_{i = 1}^nX_i$. In the case when $\{X_n\}_{n\geq 1}$ is a independent, identically distributed (iid) sequence of random variables with exponential moments, Cram\'er's Large Deviation Principle states that the tail probabilities of $\frac{S_n}{n}$ decay exponentially fast.  It is  natural to ask if this could be made more precise by finding the exact asymptotics. 

The first rigorous treatment of exact large deviation asymptotics for $S_n$ in the case when $\{X_n\}_{n\geq 1}$ is an iid sequence of random variables, was done by Cram\'er in \cite{Cr} assuming the existence of an absolutely continuous component in the distribution of $X_1$. In the a non-iid setting, in \cite{CS}, the pre--exponential factor is obtained under a decay condition on the Fourier--Laplace transform of the distribution of $X_1$. For a detailed overview of results in this direction, we refer the reader to our earlier paper \cite{FH}.

In \cite{FH}, we show that under a set of natural conditions sums of weakly dependent random variables admit asymptotic expansions for the LDP. In this paper, we extend the results in \cite{FH} by obtaining asymptotic expansions for the LDP for continuous time stochastic processes. 


\begin{definition}[Strong Asymptotic Expansions for LDP]
	Let $\{S_t\}_{t \geq 0}$ be a stochastic process with asymptotic mean zero, i.e., $\lim_{t\to \infty} \frac{\EXP(S_t)}{t}=0$. Suppose that, for some $r \in \naturals$, for each $a \in (0,L)$, the asymptotic expansion for the distribution function of $S_t$ is of the form: 
	\begin{equation}\label{strongexp}
	\PROB(S_t \geq at)e^{I(a)t} = \sum_{k=0}^{[r/2]} \frac{D_k(a)}{t^{k+1/2}}+  o_{r,a}\left(\frac{1}{t^{\frac{r+1}{2}}}\right)~~~~{\text as}~~~~~t \to \infty,\end{equation}
	where, the $I(a)$ denotes the rate function, and $D_k(a)$ are constants. Then, we refer to \eqref{strongexp} as the strong expansion for LDP of order $r$ in the range  $(0,L)$. 
\end{definition}

\begin{definition}[Weak Asymptotic Expansions for LDP]\label{LDPWeakExp}
Let $\{S_t\}_{t \geq 0}$ be a stochastic process with asymptotic mean zero. Let $(\cF, \|\cdot \|)$ be a normed space of functions defined on $\real$.  Then $S_t$ admits weak asymptotic expansion of order $r$ for large deviations in the range $(0,L)$ for $f \in \cF$ if there are functions $D^f_{k}:(0,L)\to \real$ $($depending on $f)$ for $0 \leq k < \frac{r}{2}$ such that for each $a \in (0,L)$,
	\begin{equation}\label{weakexp}\EXP(f(S_t-at))e^{I(a)t} = \sum_{k=0}^{\lfloor r/2 \rfloor} \frac{D^f_k(a)}{t^{k+1/2}}+ C_{r,a}\|f\|\cdot o\left(\frac{1}{t^{\frac{r+1}{2}}}\right),\end{equation}
where, the $I(a)$ denotes the rate function.
\end{definition}

In Section \ref{Poof}, by proving a key proposition (Proposition \ref{MainPropCts}), we show that the proofs in the discrete time can be adapted to obtain the strong expansions for LDPs for stochastic processes with weakly dependent increments.

We then apply our continuous time results to study additive functionals of diffusion processes satisfying H\"ormander's condition on a $d$--dimensional compact manifold. In Section \ref{HormanderSDE}, we show that the additive functionals of such diffusion processes have weakly dependent increments. That is, they satisfy the conditions detailed in Section \ref{Non-IID LDP} that guarantee the existence of strong expansions for LDPs. The motivation for focusing on this example comes form the work on branching diffusions in periodic media done (see \cite{HKN}), and from the large deviation problems for coupled stochastic differential equations studied in \cite{Vr1} and \cite{Li}.

Now, we make a few remarks about the relationship between the setting in \cite{HKN} and the setting here. First observe that each coordinate of the location of a particle undergoing a diffusion process in $\bbZ^d$ periodic media,  $Y_t^i$ (described in \cite{HKN}, setting the branching term equal to zero) can be viewed as an additive functional of a diffusion process on a $d-$dimensional torus. That is, suppose $X_t \in \TOR^d$ is the diffusion process generated by the following partial differential operator on~$\TOR^d$, 
\[
\cL = \frac{1}{2} \sum\limits_{ij = 1}^{d} a_{ij}(y) \frac{\partial^2 }{\partial y_i \partial y_j}
+ \sum\limits_{i = 1}^{d} b_i(y) \frac{\partial }{\partial y_i}.
\]
Then, viewing $X_t \in \TOR^d$ as taking values in $[0,1)^d\subset \real^d$, we can write $Y_t\in \real^d$ as 
\begin{equation}\label{firsty}
	dY_t^i = dX_t^i,  ~~~~~~~ Y_0 = 0,
\end{equation}
for each $1 \leq i \leq d$. Therefore, the analysis of diffusion processes in periodic media in the large deviation domain, done in \cite{HKN} to obtain the exact asymptotics for LDPs, is closely related to the question we pose in this paper. In the setting detailed in Section \ref{HormanderSDE} of this paper, we assume that $X_t$ denotes the solution of a SDE (driven by a $k-$dimensional Wiener process $W_t$) that satisfies H\"ormander's Hypoellipticity  condition (as opposed to ellipticity condition, satisfied in \cite{HKN}) on a arbitrary $d-$dimensional smooth compact manifold, and we assume that $Y_t \in \real^d$ is an additive functional of $X_t$ such that 
\begin{equation}\label{othery}
	dY_t = h(X_t) d\wt W_t + c(X_t)dt,
\end{equation}
where the Wiener process $\wt W_t$ is independent of $W_t$, $h(x)$ is non-degenerate for each $x \in M$, and $h, c$ are Lipschitz continuous. The difference between  \eqref{firsty} and \eqref{othery} is that in \eqref{othery} the Wiener process $\wt W_t$ is independent of $W_t$, while in \eqref{firsty} the process $Y_t$ and $X_t$ have the same underlying $d-$dimensional Wiener process $W_t$ (in $X_t$ it is viewed as a Wiener process on the $d-$dimensional torus while in $Y_t$ it is viewed as a Wiener process on $\real^d$). However, in this paper, under this stronger requirement of independence of the Weiner processes, we obtain higher order terms of the asymptotic expansion, as opposed to just the first term that was obtained in \cite{HKN}.

\section{Overview and main results.}\label{Non-IID LDP}
\noindent Let $\{S_t\}_{t \geq 0}$ be a stochastic process with asymptotic mean zero, i.e., \[
\lim_{t \to \infty} \frac{1}{t}\EXP(S_t)=0.
\] 
Suppose that there exists a Banach space  $\Ban$, a family of bounded linear operators $\cL(z, t):\Ban\to\Ban$, and vectors $v\in \Ban, \ell \in \Ban'$ such that  
\[
\EXP(e^{zS_t}) = \ell(\cL(z,t)v), \ t > 0,  
\]
for $z\in \complex$ for which the conditions $(D1)$ and $(D2)$ and $(D3)$ (which are detailed below) are satisfied and the family of operators $\cL(z, \cdot)$ forms a $C^0$--semigroup on the Banach space $\Ban$. That is
\[
\cL(z, t_1+t_2) = \cL(z, t_1)\circ \cL(z, t_2),~~\text{for each}~~t_1, t_2 \geq 0,~~~ \cL(z, 0) = \mathrm{Id},
\] 
and
\[
\lim_{t \to 0} \cL(z, t) = \cL(z,0) = \mathrm{Id},
\]
where the above limit is with respect to the operator norm.\vspace{5pt} \\
\textbf{\underline{Condition (D1)}}
The family of operators $\cL(z, 1 +\eta)$ satisfies the condition $[B]$ (from \cite{FH}), uniformly in $\eta \in [0, 1]$. That is,
\begin{enumerate}
	\item There exists $\delta >0$ such that the following conditions hold for all $\eta \in [0, 1]$:
	\begin{itemize}
		\item[(B1)] $z \mapsto \cL(z,1+\eta)$ is continuous on the strip $|$Re$(z)|<\delta$ and holomorphic on the disc $|z| < \delta$.
		\item[(B2)] For each $\theta \in (-\delta, \delta)$, the operator $\cL(\theta,1+\eta)$ has an isolated and simple eigenvalue  $\lambda(\theta, 1+\eta) > 0$ and the rest of its spectrum is contained\ inside the disk of radius smaller than $\lambda(\theta, 1+\eta)$ (spectral gap). In addition, $\lambda(0, 1+\eta) = 1$.
		\item[(B3)] For each $\theta \in (-\delta, \delta)$, for all real numbers $s\neq 0$, 
		the spectrum of the operator $\cL(\theta +is, 1+\eta)$, denoted by sp$(\cL(\theta +is, 1+\eta))$, satisfies: $\text{sp}(\cL(\theta +is, 1+\eta))\subseteq \{z\in \complex\ |\ |z|<\lambda(\theta, 1+\eta)\}.$
	\end{itemize}
	\item For each $\theta \in (-\delta,\delta)$, there exist positive numbers $r_1, r_2, K$ and $N_0$ such that 
	\begin{equation}\label{contD4}
	\left\Vert \cL(\theta+is,t) \right\Vert \leq \frac{\lambda(\theta)^t}{t^{r_2}}
	\end{equation}
	for all $t>N_0$, for all $K < |s| < t^{r_1} $.
\end{enumerate} 
\vspace{10pt}
\textbf{\underline{Condition (D2)}} Suppose $z \in \complex$ is such that, for all $\eta \in [0,1]$, $\cL(z, 1 +\eta)$ has an isolated simple eigenvalue $\lambda(z, 1+\eta)$. Then the projection to the top eigenspace, $\Pi(z, 1 + \eta)$, satisfies $\Pi(z, 1 + \eta)=  \Pi(z,1)$ for all $\eta \in [0,1]$. \vspace{10pt}

We denote $\Pi(\theta, 1)$ by $\Pi_{\theta}$. Using the above condition, along with the semigroup property, we conclude that for each $t>0$, the top eigenvalue of the operator $\cL(z,t)$ (whenever it exists) is equal to $\lambda(z,1)^t$.

Due to (D1), the operators $\cL(\theta,1+ \eta)$ with $\theta \in (-\delta, \delta)$ and $\eta \in [1,2]$ take the form 
\begin{equation}\label{EigenDecoCTime}
\cL(\theta, 1+\eta) = \lambda(\theta)^{1+\eta}\Pi(\theta, 1+\eta)+ \Lambda(\theta, 1+ \eta),
\end{equation}
where $\Pi(\theta, 1+ \eta)$ is the eigenprojection corresponding to the eigenvalue $\lambda(\theta)^{1+\eta}$ of the operator $\cL(\theta, 1+ \eta)$  and  $\Pi(\theta, 1+\eta)\Lambda(\theta, 1+\eta) =\Lambda(\theta, 1+\eta)\Pi(\theta, 1+\eta) =0$. Due to (D1) we can use the perturbation theory of linear operators (see \cite[Chapter 7]{Kato}) to conclude that $\lambda(\cdot)$, $\Pi{(\cdot, 1+\eta)}$ and $\Lambda{(\cdot, 1+\eta)}$ are analytic.

As a consequence of \eqref{EigenDecoCTime} and condition (D2), the family of operators $\Lambda(\theta, t)$ defined as $\cL(\theta, t) - \lambda(\theta)^t\Pi_\theta$  also forms a semigroup, and the spectral radius of the operator $\Lambda(\theta, 1)$ is less than $\lambda(\theta)$ for all $\theta \in (-\delta, \delta)$.
\vspace{10pt}\\
\textbf{\underline{Condition (D3)}} For all $\theta \in (-\delta, \delta)$, $\ell(\Pi_{\theta}v) >0$ and for all $\eta \in [0,1]$,  
\[
\frac{ \partial^2}{\partial \theta^2}\log \lambda(\theta, 1+\eta)>0.
\]
\\
\textbf{\underline{Space of functions $\fF$:}}\\ In order to state our main results, we introduce the function space $\fF_k^m$ (that are introduced in \cite{FH}) given by
\[\fF_k^m = \{f \in C^m(\real) |\ C^m_k(f) < \infty\},
\]
 where  $C^m_k(f) = \max_{0\leq j \leq m} \|f^{(j)}\|_{\text{L}^1}+  \max_{0\leq j \leq k} \|x^jf\|_{\text{L}^1}$. We call a function $f$ (left) exponential of order $\alpha$, if $\lim_{x\to -\infty} |e^{-\alpha x} f(x)| = 0$.
Define the function space $\fF^m_{k,\alpha}$ by 
\[\fF^m_{k,\alpha} = \{f \in \fF_k^m|\ f^{(m)}\text{ is exponential of order}\ \alpha\}.\]
It is clear that $\fF^m_{k,\alpha} \subset \fF^m_{k, \beta}$ if $ \alpha>\beta$.  Finally, define, $\fF^m_{k,\infty} = \bigcap_{\alpha>0} \fF^m_{k, \alpha}$. \\
\\
The following proposition, which will be proved in \ref{Poof}, is the key idea in adapting the proofs of discrete time results from \cite{FH} to continuous time. 
\begin{proposition}\label{MainPropCts}
	Suppose that the conditions $(D1)$ and $(D2)$ hold. Then, for a fixed $\theta \in (-\delta, \delta)$, there exists $\tilde{\delta} > 0$ such that, for each $s \in (-\tilde{\delta}, \tilde{\delta})$, for each $t \geq 1$, the operator $\cL({\theta+is}, t)$ has a simple top eigenvalue $\lambda(\theta+is)^{t}$ and
	\begin{equation}\label{CtsEigenDec}
	\cL({\theta+is}, t) =\lambda(\theta+is)^t\Pi_{\theta+is}+ \Lambda(\theta+is, t),
	\end{equation}
	where $\Pi_{\theta+is} \equiv \Pi(\theta+is, t)$ is the eigenprojection corresponding to the eigenvalue $\lambda(\theta+is)^{t}$ and  $\Pi(\theta+is, t)\Lambda(\theta+is, t) =\Lambda(\theta+is, t)\Pi(\theta+is, t) =0$. In addition, the family of operators $\{\Lambda(\theta+is, t)\}_{t\geq 1}$ satisfies $\Lambda(\theta+is, tN) = \Lambda(\theta+is, t)^N$ for all $t\geq 1, N \in \naturals$ and the spectral radius of the operator $\Lambda(\theta + is, 1)$ is less than $|\lambda(\theta + is)|$.
\end{proposition}

The following theorems are the continuous time analogues of the discrete time results, Theorem 2.1, Theorem 2.2 and Theorem 2.3 from \cite{FH},  respectively.  We do not repeat the proofs of Theorems \ref{WeakExpContTime}, \ref{StrongExpContTime} and \ref{FirstTermContTime} in our current continuous time setting, since the proofs are completely analogous to those in \cite{FH}. The crucial point, however, is that the continuous time results require the use of Proposition \ref{StrongExpContTime} which we prove in the next section.
\\
\begin{theorem}\label{WeakExpContTime} Let $r \in \naturals$. Suppose that conditions $(D1), (D2)$ and $(D3)$ hold. Then, for all $a\in \Big(0, \frac{\log{\lambda(\delta)}}{\delta}\Big)$, there exist $\theta_a \in (0, \delta)$  and polynomials $P^a_{k}(x)$ of degree at most $2k$, such that for $q > \frac{r+1}{2r_1}+1$ and $\alpha > \theta_a$, for all $f \in \fF^{q}_{r+1, \alpha}$
	\[\EXP(f(S_t-at))e^{I(a)t} = \sum_{k = 0}^{\lfloor r/2 \rfloor} \frac{1}{t^{k+1/2}}\int  P^a_{k}(x) f_{\theta_a}(x)\, dx+ C^q_{r+1}(f_{\theta_a})\cdot o_{r,a}\left(\frac{1}{t^{\frac{r+1}{2}}}\right)~~~~~~{\text as}~~~~t \to \infty,\]
	where $f_{\theta}(x)=\frac{1}{2\pi} e^{-\theta x}f(x)$ and $I(a)=\sup_{\theta\in (0,\delta)}[a\theta-\log \lambda(\theta)]=a\theta_a-\log \lambda(\theta_a).$
\end{theorem}
\begin{theorem}\label{StrongExpContTime}
	Let $r \in \naturals$, $r \geq 2$. Suppose that conditions $(D1), (D2)$ and $(D3)$ hold with $r_1 > r/2$.  Then, for each $a\in \Big(0, \frac{\log{\lambda(\delta)}}{\delta}\Big)$, there exist constants $D_k(a)$ such that 
	\begin{equation*}
	\PROB(S_t \geq at)e^{I(a)t} = \sum_{k=0}^{[r/2]} \frac{D_k(a)}{t^{k+1/2}}+  o_{r,a}\left(\frac{1}{t^{\frac{r+1}{2}}}\right)~~~~{\text as}~~~~~t \to \infty,\end{equation*}
	where, the rate functional $I(a)$ is defined as 
	\[
	I(a): =\sup_{\theta\in (0,\delta)}[a\theta-\log \lambda(\theta,1)]=a\theta_a-\log \lambda(\theta_a, 1).\]
\end{theorem}
The following theorem shows that, under a set conditions weaker than those required in the above two theorems (namely, without requiring the condition $(D1)-(2)$), the exact asymptotics for the LDP can be obtained (that is, the first term of the asymptotic expansion, including the pre-exponential factor). 
\begin{theorem}\label{FirstTermContTime}
	Suppose that $(D1)-(1), (D2)$ and $(D3)$ hold. Then, for each $a \in \Big(0,\frac{\log{\lambda(\delta)}}{\delta}\Big)$, 
	\[\PROB(S_t \geq at)e^{I(a)t} =  \frac{\ell(\Pi_{\theta_a} v)\sqrt{I''(a)}}{\theta_a\sqrt{2\pi t}}\Big(1 + o(1)\Big)\,\,\,\,\,{\text as}\,\,t \to \infty.\]
\end{theorem}

\section{Proofs of the main results}\label{Poof}

\begin{proof}[Proof of Proposition \ref{MainPropCts}]
	Let $\theta \in (-\delta, \delta)$ and $\eta \in [0,1]$ be fixed. Consider the two parameter perturbation of the operator $\cL(\theta, 1+\eta)$ of the form $\cL(\theta + is, 1+ \eta + \eps)$. From condition (D1), for a fixed $\eta$, $z \mapsto \cL(z,1+\eta)$ is holomorphic on the disc $|z| < \delta$ and for each fixed $z$, the family of operators $\cL(z,t)$ forms a $C^0$--semigroup. In addition, the two parameter operator $\cL(z,t)$ is uniformly bounded on the region $\{(z,t): |z| < \delta, t\in [1,2]\}$. From here, using the Cauchy integral formula for analytic functions it is clear to see that this two parameter perturbation is continuous. Hence, by perturbation theory, for each $\eta \in [0,1]$, there exists $\delta_\eta >0$ such that, on the set $\{(s, \eps): |s| < \delta_\eta, \eps < \delta_\eta \}$, 
	\[
	\cL({\theta+is},  1+ \eta + \eps) =\lambda(\theta+is, 1+ \eta + \eps)\Pi(\theta+is,  1+ \eta + \eps) + \Lambda(\theta+is,  1+ \eta + \eps),
	\]
	where  $\Pi(\theta+is,  1+ \eta + \eps)$ is the projection on the top eigenfunction of the operator $\cL(\theta+is,  1+ \eta + \eps)$ corresponding to the simple top eigenvalue $\lambda(\theta+is, 1+ \eta + \eps)$ and  
	\[\Pi(\theta+is,  1+ \eta + \eps)\Lambda(\theta+is,  1+ \eta + \eps) =\Lambda(\theta+is,  1+ \eta + \eps)\Pi(\theta+is, 1+ \eta + \eps) =0.
	\]
	In addition, the spectral radius of $\Lambda(\theta + is,  1+ \eta + \eps)$ is less than $|\lambda(\theta + is,1+ \eta + \eps)|$.
	
	Since the interval $[0,1]$ is compact, we can choose $\eta_1, \eta_2,\cdots, \eta_k$ such that the set $\{\eta: |\eta - \eta_i| < \delta_{\eta_i}, i = 1, 2, \cdots k\}$ contains the interval $[0,1]$. Put $\tilde{\delta} = \min\limits_{i = 1,2, \cdots k} \delta_{\eta_i}$. Thus, for all $\eta \in [0,1]$ and $s$ such that $|s| < \tilde{\delta}$, 
	\[
	\cL({\theta+is},  1+ \eta) =\lambda(\theta+is,1+ \eta)\Pi(\theta+is,  1+ \eta) + \Lambda(\theta+is,  1+ \eta),
	\]
	and the  spectral radius of $\Lambda(\theta + is,  1+ \eta )$ is less than $|\lambda(\theta + is,1+ \eta)|$.
	
	Put $\Pi_{\theta+ is} =  \Pi(\theta+ is,1)$. From (D2) we know that $\Pi(\theta + is,  1+ \eta) = \Pi_{\theta+ is} $ for all $\eta \in [0,1]$ and $|s| < \tilde{\delta}$.  This, along with the semigroup property of the operators $\cL(\theta + is,t)$, implies that
	$\lambda(\theta + is,1+ \eta) = \lambda(\theta + is)^{1+\eta}$ for all for all $\eta \in [0,1]$, $|s| < \tilde{\delta}$.  To see this, first note that we do not assume that the top eigen-value for the operator $\cL(\theta + is,\eta)$ exists for $\eta \in [0,1)$. Now, if $\eta$ is rational, we have $\eta = p/q$ for some $p,q \in \naturals, q \neq 0$. Let $v(\theta + is) \in \Ban$ be a non-zero vector be such that $ \Pi(\theta + is,  1+ \eta)v(\theta + is) = \Pi_{\theta+ is}v(\theta + is) = v(\theta + is)$ for all $\eta \in [1,2]$. Then we have, 
	\begin{align*}
	\lambda(\theta + is)^{q+p} v(\theta + is) &= \cL(\theta+is, 1)^{q+p}v(\theta + is)\\
	&= \cL(\theta+is, q+p)v(\theta + is) \\
	&= \cL(\theta+is, 1+p/q)^q v(\theta + is)\\ &= \lambda(\theta + is, 1+p/q)^q v(\theta + is).
	\end{align*} 
	
	Therefore, $\lambda(\theta + is)^{1+\eta} = \lambda(\theta + is, 1+\eta)$ for all rational $\eta \in [0,1]$. Since, the semigroup $\cL(\theta + is,t)$ is continuous in $t$, we have that the top eigenvalue $\lambda(\theta + is, 1+\eta)$ is continuous in $\eta$, and therefore, the relation  $\lambda(\theta + is)^{1+\eta} = \lambda(\theta + is, 1+\eta)$ holds for all $\eta \in [0,1]$.

	For $t \geq 1$, define the new family of operators $ \Lambda(\theta+is, t) = \cL({\theta+is},t) - \lambda(\theta+is)^{t}\Pi_{\theta+is}$. 
	It is clear to see from this definition that $\Lambda(\theta+is, tN) = \Lambda(\theta+is, t)^N$ for all $t\geq 1, N \in \naturals$. Then, using the fact that $\frac{t}{[t]} \in [1,2]$, we have
	\begin{align*}
	\cL({\theta+is},  t) =\cL\Big({\theta+is},  \frac{t}{[t]}\Big)^{[t]}  &= \Big(\lambda(\theta+is)^{\frac{t}{[t]}}\Pi_{\theta+is} + \Lambda\Big(\theta+is, \frac{t}{[t]}\Big)\Big)^{[t]}\\& =\lambda(\theta+is)^t\Pi_{\theta+is} + \Lambda(\theta+is, t).
	\end{align*}
	Here, the spectral radius of the operator $\Lambda(\theta+is, 1)$ is less than $|\lambda(\theta+is)|$. This concludes the proof of Proposition \ref{MainPropCts}.
\end{proof}

\begin{remark}
	Our equation \eqref{CtsEigenDec} is the continuous time analogue of equation  $(2.2)$ from \cite{FH}. This, along with assumption $(D3)$, allows us to obtain proofs of Theorems \ref{WeakExpContTime}, \ref{StrongExpContTime} and \ref{FirstTermContTime} by  replacing the discrete time steps $n$ by $t\in \real^+$ and replacing  $\overline\cL^n_s$ by $$\overline\cL(s,t)=\frac{e^{-iast}}{\lambda(\theta)^t}\cL(\theta_a+is,t)$$ in  the proofs of the corresponding discrete time results from \cite{FH}. 
\end{remark}

\section{SDEs satisfying H\"ormander Hypoellipticity condition}\label{HormanderSDE}

Let $M$ be a compact $d-$ dimensional smooth manifold and $\{V_0,\dots,V_k\}$ be a collection of smooth vector fields of $M$ such that $D=\{V_1, \dots V_k\}$ satisfies the H\"ormander Hypoellipticity condition, i.e.,\hspace{3pt}the Lie algebra generated by $D$ evaluated at $x$ spans the tangent space $T_xM$ at each $x \in M$.  

Let $W_t$ be the $k-$dimensional Wiener process with components $W^i_t$ for $1 \leq i \leq k$. Let $X_t$ be the process on $M$, and $Y_t$ be the process on $\mathbb R$ satisfying the coupled SDEs, 
\begin{equation}\label{SDE}
dX_t = \sum_{i=1}^k V_i(X_t) \circ dW^{i}_t + V_0(X_t)\, dt, ~~~~~~~~ X_0 = x,
\end{equation}
\begin{equation}\label{FunctionalEq}
dY_t = \sigma(X_t)\circ d\wt{W}_t + b(X_t)\,dt,  ~~~~~~~ Y_0 = y,
\end{equation}
where the real valued function $b:M \to \mathbb R$ and the real valued function $\sigma: M \to \mathbb R$ are smooth and $\wt{W}_t$ is a $1-$dimensional Wiener process  independent of the $k-$dimensional Wiener process $W_t$, and $\sigma$ is non-degenerate, i.e, $\sigma^2(x) >0$ for each $x \in M$.. The right hand sides of $(\ref{SDE})$ and \eqref{FunctionalEq} are interpreted in the Stratonovich sense. Observe that, in  \eqref{FunctionalEq}, it is equivalent to consider the It\^o or the Stratonovich sense, since the coefficient $\sigma(X_t)$ of the Wiener process $\wt{W}_t$ is independent of $Y_t$.  Note that the distribution of $X_t$ for each $t > 0$ is absolutely continuous by H\"ormander's theorem.

\begin{theorem}
Under the above assumptions, for all $r \in \naturals \cup\{0\}$, 
		 \begin{enumerate}
			\item[$(a)$] $Y_t$ admits the weak expansion of order $r$ in the range $(0,\infty)$ for $f \in \fF^q_{r+1,\alpha}$ with $q \geq 1$ and suitable $\alpha$ depending on $a$ and
			\item[$(b)$] $Y_t$ admits the strong expansion of order $r$ in the range $(0,\infty)$.
		\end{enumerate}

\end{theorem}

\begin{proof}
	The infinitesimal generator of the joint Markov process $(X_t, Y_t)$ is a partial differential operator $\cM$ acting on functions $u$ defined on $M \times \real$ given by
	\begin{equation}\label{generator}
	\cM u = \frac{1}{2} \nabla_x [(V(x)V^T(x))\nabla_x u] + \frac{1}{2} (\sigma^2(x))\Delta_y u + V_0(x)\nabla_x u + b(x)\nabla_y u,
	\end{equation}
	where $V(x)$ is the $d \times k$ matrix formed by the vectors $\{V_1, \dots V_k\}$ as columns.
	
	Let $\bar\rho(x)$ be the invariant density of the process $X_t$ on $M$, that is, $\bar\rho(x)$ is the density of a measure defined on $M$, satisfying \[\cM^* \bar \rho = 0, ~~~~~~~~~~~~~~~~~~~~~~ \int_M \bar{\rho} = 1.\] We assume that $$\int_M b(x)\, d\bar\rho(x) = 0.$$
	The above condition guarantees that the asymptotic mean of the random process $Y_t$ is zero, since
	\[
	\bar{Y} = \lim_{t \to \infty} \frac{1}{t} \EXP(Y_t ) = \lim_{t \to \infty} \frac{1}{t} \EXP\Big(\int_0^t b(X_s)\,ds \Big) = \int_M b(x)\, d\bar\rho(x) 
	\]
	We also observe that, from the Kolmogorov Forward Equation, the transition density for the Markov process $(X_t,Y_t)_{t \geq 0}$ is given by $p(t,(x_0,y_0),(x,y))$, and it satisfies the PDE 
	\begin{equation}\label{transden}
	\begin{aligned}
	\partial_t p &= \cM_{(x,y)}^* p, \\
	p(0,(x_0,y_0),&(x,y)) = \delta_{(x_0,y_0)}(x,y).
	\end{aligned}
	\end{equation}	
	Let $\Ban$ be the Banach space of complex valued continuous functions defined on $M$ equipped with the supremum norm. Define, for each $z \in \complex$, $t \geq 0$, the bounded linear operator  $\cL(z,t) : \Ban \to \Ban$ given by
	\[
	\cL(z,t)f(x) = \EXP_{(x,y)}(f(X_t)e^{z(Y_t-y)}),
	\]
	where the right hand side clearly does not depend on $y$. That is, for the constant function $v = \bf{1} \in \Ban$, and the measure $\ell = \delta_x \in \Ban'$ (the space of bounded linear functionals on $\Ban$) we have 
	\begin{equation}
	\label{MainAssumMarkov}
	\EXP_{(x,0)}\left(e^{z Y_t}\right)=\ell(\cL(z,t) v),
	\end{equation}
	The family of operators $\{\cL(z,t)\}_{\{t \geq 0\}}$ forms a semigroup since 
	\begin{align*}
	\cL(z,t)\circ \cL(z,s)f(x) &=  \EXP_{(x,y)}((\cL(z,s)f)(X_t)e^{z(Y_t-y)})\\
	&= \EXP_{(x,y)}( \EXP_{(X_t,Y_t)}(f(X_s)e^{z(Y_s-Y_t)})e^{z(Y_t-y)})\\
	& = \EXP_{(x,y)}( \EXP_{(X_t,Y_t)}(f(X_s)e^{z(Y_s-y)}))\\
	& =  \EXP_{(x,y)}(f(X_{s+t})e^{z(Y_{s+t}-y)})\\
	& = \cL(z, t+s)f(x).
	\end{align*}
	Now we will verify conditions $(D1)$, $(D2)$ and $(D3)$ from Section \ref{Non-IID LDP} for the family of operators $\cL(z,t)$. To verify condition $(D1)$, we will show that $(B1)-(B3)$ hold uniformly on $t \in [1,2]$ and show that \eqref{contD4} holds.\\
	\\
	\underline{\bf Condition (B1)} We first observe that the map $z \mapsto \cL(z,t)$ is infinitely differentiable in $z$ for all $z \in \complex$. Indeed, for each $f \in \Ban$, $\alpha \in \mathbb{Z}_+$, and $z \in \complex$, $D^\alpha _z(\cL(z,t)f)(x_0) = \EXP_{(x_0,0)}(Y_t^\alpha f(X_t)e^{zY_t})$. We know that $Y_t$ is a stochastic process on $\real$ with bounded diffusion and drift coefficients, which implies that $Y_t$ has all exponential moments. Hence,   $D^\alpha _z\cL(z,t)$ is a well defined bounded linear operator on $\Ban$ for all $\alpha \in \mathbb{Z}_+$ and $z \in \complex$.
	
	Note that $\cL(0,t)$ is a compact operator on $\Ban$ since, if we define $$q_{0,t}(x_0,x) = \int_{\real} p(t,(x_0,0),(x,y))dy,$$ then, for any $f \in \Ban$, $\cL(0,t)f(x_0) = \int_{M}f(x)q_{0,t}(x_0,x)dx$, where $q_{0,t}$ is positive and continuous in $(x_0,x) \in M \times M$. We note that $1$ is the top eigenvalue of $\cL(0,t)$ with constant functions forming the eigenspace. All the other eigenvalues of $\cL(0,t)$ have absolute values less than 1, by the Perron--Frobenius theorem.

	We note that if $\theta \in \real$, then $q_{\theta,t}(x_0,x) = \int_{\real} e^{\theta y} p(t,(x_0,0),(x,y))dy >0$ for all $x_0,x \in M$. This kernel is  continuous in $(x_0,x) \in M \times M$. That is, $\cL(\theta,t)$ is a positive, compact operator for all $\theta \in \real$. Thus Condition (B1) is satisfied uniformly on $t \in[1,2]$.\\
	\\
	\underline{\bf Condition (D2):} We observe that the coefficients of the operator $\cM$ are independent of the time variable $t$, and therefore the Markov process $(X_t, Y_t)$ is time homogeneous. Thus, the top eigenspace of the operators $\cL(\theta,t)$ is the same for all $t >0$. Thus, $\Pi(\theta, t) = \Pi(\theta,1)$ for all $t >0$, in particular, Condition (D2) is satisfied. \\
	\\
	\underline{\bf Condition (B2)} Using (D2) and the semigroup property, condition (B2) is satisfied since there exists a $\lambda(\theta)>0$ for all $\theta$, the top eigenvalue $\lambda(\theta)^t$ of the operator $\cL(\theta,t)$ exists, and other eigenvalues of $\cL(\theta,t)$ have absolute values less than $\lambda(\theta)^t$.\\
	\\
	\underline{\bf Condition (B3)} We need to show that we have sp$(\cL(\theta + is,t) )\subseteq \{|z|<\lambda(\theta)^t\}$. We first note that 
	\[
	|\cL(\theta + is,t) f(x)| = |\EXP_{(x,y)}( f(X_t)e^{(\theta + is)(Y_t - y)})| \leq  \EXP_{(x,y)}( |f(X_t)e^{(\theta + is)(Y_t - y)}|) 
	\]
	\[
	= \EXP_{(x,y)}( |f(X_1)|e^{\theta (Y_1-y)}) =  \cL(\theta,t)|f|(x).
	\]
	Thus sp$(\cL(\theta + is, t)) \subseteq \{|z|\leq \lambda(\theta)^t\}$. To prove that there is inclusion with strict inequality, using the fact that the top eigenvalue of the operator $\cL(\theta,t)$ is $\lambda(\theta)^t$, it is enough to show that sp$(\cL(\theta + is, 1)) \subseteq \{|z|< \lambda(\theta)\}$. We suppose, on the contrary, that there exists an eigenfunction $f \in \Ban$ of the operator $\cL(\theta + is, 1)$, with $\|f\| = 1$ corresponding to the eigenvalue $\lambda(\theta+is)$ such that $|\lambda(\theta + is)| = \lambda(\theta)$. That is, for all $x \in M$,
	\begin{equation}\label{eigenfunctionf}
	\EXP_{(x,0)}( f(X_1)e^{(\theta + is)Y_1}) = \lambda(\theta + is)f(x).
	\end{equation}
	We know $\lambda(\theta)$ is the top eigenvalue of the operator $\cL(\theta,1)$. Thus, there exists an eigenfunction $g \in \Ban$ of $\cL(\theta,1)$, corresponding to the eigenvalue $\lambda(\theta)$, which implies 
	that for all $x \in M$,
	\begin{equation}\label{eigenfunctiong}
	\EXP_{(x,0)}( g(X_1)e^{\theta Y_1}) = \lambda(\theta)g(x).
	\end{equation}
	Note that we can assume that, for all $x \in M$, $g(x) > 0$, and that $|f(x)| \leq g(x)$. In addition, we can assume that there exists a point $x_0 \in M$ such that $|f(x_0)| = g(x_0)$.
	Now,
	\[
	|\EXP_{(x_0,0)}( f(X_1)e^{(\theta + it)Y_1})| = |\lambda(\theta)f(x_0)| = \lambda(\theta)g(x_0) = \EXP_{(x_0,0)}( g(X_1)e^{\theta Y_1}).
	\]
	Thus,
	\[
	\EXP_{(x_0,0)}( |f(X_1)e^{(\theta + it)Y_1}|) \geq \EXP_{(x_0,0)}( g(X_1)e^{\theta Y_1}). 
	\]
	This implies that 
	\[
	\EXP_{(x_0,0)}( e^{\theta Y_1}(|f(X_1)e^{itY_1}| - g(X_1))) \geq 0,
	\]
	and therefore,
	\[
	\EXP_{(x_0,0)}( e^{\theta Y_1}(|f(X_1)| - g(X_1))) = \cL(\theta,1)(|f| - g)(x_0) \geq 0.
	\]
	We have from our assumption that $|f| \leq g$, and we know that $\cL(\theta,1)$ is a positive operator. We conclude that, 
	\[
	\cL(\theta, 1)(|f| - g)(x_0) = \EXP_{(x_0,0)}( e^{\theta Y_1}(|f(X_1)| - g(X_1))) = 0.
	\]
	Now, 
	\[
	\EXP_{(x_0,0)}( e^{\theta Y_1}(|f(X_1)| - g(X_1))) = \int_{M} (|f(x)| - g(x))q_{\theta,1}(x_0,x) dx.
	\]
	From the definition of $q_{\theta,1}$, we know that, for a fixed $x_0 \in M$, $q_{\theta,1}(x_0, x) >0$, $x \in M$. Therefore, for all $x \in M$, $|f(x)| = g(x)$. Thus, there exists a continuous function $\phi$ defined on $M$ such that $f(x) = e^{i\phi(x)}g(x)$ for all $x \in M$. Substituting this in $(\ref{eigenfunctionf})$, we get 
	\[
	\EXP_{(x,0)}( e^{i\phi(X_1)}g(X_1)e^{(\theta + is)Y_1}) = \lambda(\theta +is)e^{i\phi(x)}g(x)\]\[ = e^{i\phi(x)}\EXP_{(x,0)}( g(X_1)e^{\theta Y_1})\frac{\lambda(\theta+ is)}{\lambda(\theta)},
	\]
	where the last equality follows from equation $(\ref{eigenfunctiong})$. In addition, since $|\lambda(\theta + is)| = \lambda(\theta)$, there exists a constant $c$ such that $\frac{\lambda(\theta+ is)}{\lambda(\theta)} = e^{ic}$.
	Therefore,
	\[
	\EXP_{(x,0)}( e^{i\phi(x)}e^{\theta Y_1}e^{ic}g(X_1)(e^{isY_1+ i\phi(X_1) - i\phi(x)-ic} -1)) = 0.
	\]
	This implies that whenever $p(1, (x,0), (\tilde{x}, \tilde{y})) >0$, \[s\tilde{y} + \phi(\tilde{x}) - \phi(x) -ic = 0~~~~ \text{(mod $2\pi$}).\] This is impossible  since the Brownian motion $\wt{W}$ (in the definition of $Y_1$) is independent of ${W}$ (in the definition of $X_1$). Thus, sp$(\cL(\theta + is,1)) \subseteq \{|z|<\lambda(\theta)\}$, which implies sp$(\cL(\theta + is,t)) \subseteq \{|z|<\lambda(\theta)^t\}$.\\
	\\
	\\
	\underline{\bf Condition (D2)-2 } Let $\theta \in \real$ be fixed. Let $g_\theta(x)$ be such that $\|g_\theta\| = 1$ and $\cL(\theta, 1)g_\theta(x) = \lambda(\theta)g_\theta(x)$ for all $x \in M$. Then we also have $\cL(\theta, t)g_\theta(x) = \lambda(\theta)^t g_\theta(x)$ for all $x \in M$, since condition (D2) holds. In addition, since  $\cL(\theta, 1)$ is a positive operator, the eigenfunction $g_\theta$ is positive. We observe that $g_\theta$ satisfies the PDE  $e^{-\theta y}\cM(e^{\theta y}g_\theta(x)) = \mu(\theta)g_\theta(x)$ for all $x \in M$, $y \in \real$, where $\mu(\theta) = \log\lambda(\theta)$. Since the coefficients of the operator $e^{-\theta y}\cM(e^{\theta y} \cdot)$ are differentiable in $\theta$, the function $g_\theta$ is differentiable in $\theta$.  
	
	We first consider a new family of operators $\tilde{\cL}(z,t): \Ban \to \Ban$ defined by 
	\[
	\tilde{\cL}(z,t)f(x_0) = \int_M f(x)\tilde{q}_{z,t}(x_0,x) \, dx,
	\]
	where $\tilde{q}_{z,t}(x_0,x) = \int_\real e^{zy}  p_\theta(t, (x_0, 0),(x,y))\, dy$ and
	\[
	p_\theta(t, (x_0, 0),(x,y)) := \frac{e^{\theta y}g_\theta(x)p(t,(x_0,0), (x,y))}{\lambda(\theta)^t g_\theta(x_0)}.
	\]
	Let $\bf{1}$ denote the function that takes the value $1$ for all $x_0 \in M$. Note that 
	\begin{align*}
	\tilde{\cL}(0,t){\bf{1}}(x_0) & = \int_M 1 \cdot  \tilde{q}_{0,t}(x_0,x) \, dx \\ &  =  \int_M \int_\real p_\theta(t, (x_0, 0),(x,y))\, dy \, dx\\ & = \int_M \int_\real \frac{e^{\theta y}g_\theta(x)p(t,(x_0,0), (x,y))}{\lambda(\theta)^t g_\theta(x_0)}\, dy \, dx\\ & = \frac{1}{\lambda(\theta)^t g_\theta(x_0)}\int_M \int_\real e^{\theta y}g_\theta(x)p(t,(x_0,0), (x,y))\, dy \, dx\\& 
	= \frac{1}{\lambda(\theta)^t g_\theta(x_0)} \cL(\theta, t)g_\theta(x_0) = 1. 
	\end{align*}
	Hence, $\bf{1}$ is an eigenfunction for the operator $\tilde{\cL}(0,t)$ corresponding to the top eigenvalue $1$.
	
	Observe that the operators $\tilde{\cL}$ and $\cL$ satisfy, for all $f \in \Ban$, 
	\[
	\tilde{\cL}(z,t)f(x_0) = \frac{1}{\lambda(\theta)^t g_\theta(x_0)} \cL(\theta+z,t)(fg_\theta)(x_0).
	\]
	It is easy to see that the new family of operators $\{\tilde{\cL}(z,t)\}_{t\geq 0}$ also forms a $C_0$ semigroup.
	Thus, in order to prove \eqref{contD4}, we need to show that there exist positive numbers $r_1, r_2, K$ and $N_0$ such that 
	\begin{equation*}
	\| \tilde{\cL}(is,t) \|  \leq \frac{1}{t^{r_2}}
	\end{equation*}
	for all $t>N_0$, for all $K < |s| < t^{r_1} $.
	In fact, it will be enough to show that there exists an $\epsilon \in (0,1)$ such that, for all $t \in [1,2]$ and for all $|s| >K$, 
	\begin{align}\label{expdecay}
	\| \tilde{\cL}(is,t) \|  < 1-\epsilon,
	\end{align}
	since the above relation would imply that, for all $t >2$,
	\begin{align*}
	\|\tilde{\cL}(is,t)\| = \left\Vert\tilde{\cL}\Big(is,\frac{t}{[t]}\Big)^{[t]}\right\Vert 
	\leq \left\Vert \tilde{\cL}\Big(is,\frac{t}{[t]}\Big)\right\Vert^{[t]}
	\leq (1-\epsilon)^{[t]}.
	\end{align*}
	showing exponential decay.
	
	We observe that for any $f \in \Ban$, and $x_o \in M$,
	$$\tilde{\cL}(is,t)f(x_0) = \int _M f(x) \tilde{q}_{is,t}(x_0,x) \, dx$$ where,
	$$ \tilde{q}_{is,t}(x_0,x) = \int_{\real} \frac{e^{(\theta + is) y}g_\theta(x)p(t,(x_0,0), (x,y))}{\lambda(\theta)^t g_\theta(x_0)}\, dy,$$ ans therefore, it is enough to show that there exists an $\epsilon \in (0,1)$ and $K > 0$  such that for all $|s| >K$, and for all $t \in [1,2]$, 
	\begin{equation}\label{qleseps}
	|\tilde{q}_{is,t}(x_0,x)| \leq 1 - \epsilon.
	\end{equation}
	Let $\cF_t$ denote the sigma algebra generated by the process $\{W_u\}_{u \in [0,t]}$.  Note that the following equality holds, 
	\begin{align*}
	&\tilde{q}_{is,t}(x_0,x) =  \frac{g_\theta(x)}{\lambda(\theta)^tg_\theta(x_0)}\EXP_{(x_0,0)}\big(e^{(\theta+is)Y_t}\big| X_t = x\big) \\
	&=  \frac{g_\theta(x)}{\lambda(\theta)^tg_\theta(x_0)}\EXP_{(x_0,0)}\Big(\EXP(e^{(\theta + is)\big(\int_0^t \sigma(X_u) \,d\wt{W}_u+ \int_0^t b(X_u)\, du \big)}\Big|\cF_t\Big)\Big| X_t = x\Big)\Big)
	\end{align*}
	We know that $\Big\{e^{\int_0^t (\theta + is)\sigma(X_u) \,d\wt{W}_u - \frac{1}{2}\int_0^t (\theta + is)^2\sigma^2(X_u) \,du}\Big| \cF_t\Big\}$ forms a martingale for all $t >0$. Therefore, 
	\begin{align*}
	&\EXP(e^{(\theta + is) Y_t}|\cF_t) =  \EXP(e^{(\int_0^t (\theta + is)^2\sigma^2(X_u) \,du+ (\theta + is)\int_0^t b(X_u)\, du )}|\cF_t) \\
	&= \EXP\Big(e^{ (\theta^2\int_0^t \sigma^2(X_u) \,du -s^2 \int_0^t \sigma^2(X_u) \,du+ 2is\theta\int_0^t \sigma^2(X_u) \,du +(\theta + is)\int_0^t b(X_u)\, du )}\Big|\cF_t\Big).
	\end{align*}
	Let $\epsilon \in (0,1)$. Since $\sigma(x)$, $b(x)$ are smooth on the compact manifold $M$, and $\sigma(x)> 0$ for all $x \in M$, for a fixed $\theta >0$, we can choose $K>0$ such that for all $t \in [1,2]$, $|s| > K$,
	\begin{align*}
	&\Big|\EXP(\exp\Big((\theta^2\int_0^t \sigma^2(X_u) \,du -s^2 \int_0^t \sigma^2(X_u) \,du+\\&~~~~~~~~~~~~~~~~~~~~~~~+  2is\theta\int_0^t \sigma^2(X_u) \,du+ (\theta + is)\int_0^t b(X_u)\, du )\Big)|\cF_t)\Big| \\ &~~~~~~~~~~~~~~~~~~~~~< (1-\epsilon) \frac{\|g_\theta\|\sup\{\lambda(\theta)^t|t \in [1,2]\}}{\inf\{g_\theta(x)| x \in M\}}.
	\end{align*}
	Note that the quantities $\sup\{\lambda(\theta)^t\ |\ t \in [1,2]\}$ and $\inf\{g_\theta(x)\ |\ x \in M\}$ are strictly positive and finite due to condition (B2) and the fact that eigenfunction $g_\theta$ is strictly positive on $M$.
	Therefore,
	\begin{align*}
	&\Big|\EXP_{(x_0,0)}(e^{(\theta + is) Y_t} | X_t = x)\Big| =
	\Big|\EXP_{(x_0,0)}\Big(e^{(\theta + is) \big(\int_0^t \sigma(X_u) \,d\wt{W}_u+ \int_0^t b(X_u)\, du \big)}\Big|\cF_t\Big)\Big| X_t = x\Big)\Big|\\& \leq \EXP_{(x_0,0)}\Big(\Big|\Big(\EXP(e^{(\theta + is) (\int_0^t \sigma(X_u) \,d\wt{W}_u+ \int_0^t b(X_u)\, du) }\Big|\cF_t\Big)\Big|\ \Big| X_t = x\Big) \\& \leq  (1-\epsilon) \frac{\|g_\theta\|\sup\{\lambda(\theta)^t|t \in [1,2]\}}{\inf\{g_\theta(x)| x \in M\}},
	\end{align*}
	As a result $|\tilde{q}_{{is},t}(x_0,x)| \leq (1 - \epsilon)$. This implies that for all $t \in [1,2]$, $|s| > K$, $\| \tilde{\cL}(is,t) \|  < 1-\epsilon$, which concludes the proof of condition (D1).\\
	\\
	\\
	\underline{\bf Condition (D3):} First, observe that $\ell(\Pi_\theta v) = \delta_x(\Pi_{g_\theta} {\bf{1}}) = g_\theta(x)\int_M g_\theta >0$. Now, that the top eigenvalue of operators $\cL(z,1+\eta)$ is $\lambda(\theta)^{1+\eta}$. Thus, it is enough to show that $\log\lambda(\theta)$ is twice continuously differentiable and the second derivative is positive for all $\theta \in \real$. Let $\mu(\theta) = \log\lambda(\theta)$.
	
	Let $\theta> 0$ be fixed. We know that the function $g_\theta$ is such that \begin{equation}\label{gthe}
	\cL(\theta,t)g_\theta = e^{t\mu(\theta)}g_\theta.
	\end{equation} 
	Let $\psi_\theta$ be a linear functional in $\Ban'$ satisfying 
	$\langle \psi_\theta, \cL(\theta,t) f \rangle = e^{t\mu(\theta)}\langle \psi_\theta, f\rangle$ for all $f \in \Ban$, and $\langle \psi_\theta, g_\theta\rangle = 1$. Let us define a new operator $\cL'(\theta, t)$, which is the derivative of the operator $\cL(\theta, t)$ with respect to $\theta$. Thus, 
	$$\Big(\cL'(\theta, t)f\Big)(x_0) = \EXP_{(x_0, 0)}(f(X_t)Y_te^{\theta Y_t}).$$
	We differentiate equation \eqref{gthe} on both sides with respect to $\theta$ to obtain
	\begin{align}
	\cL'(\theta, t)g_\theta(x_0) + \cL(\theta, t)g_\theta'(x_0) & =  \EXP_{(x_0, 0)}(g_\theta(X_t)Y_te^{\theta Y_t}) +  \cL(\theta, t)g_\theta'(x_0)\nonumber\\
	& = \label{1stder} t\mu'(\theta)e^{t\mu(\theta)}g_\theta(x_0) + e^{t\mu(\theta)}g'_\theta(x_0).
	\end{align}
	Therefore, applying the linear functional $\psi_\theta$ on both sides, we obtain,
	\begin{align*}
	\langle \psi_\theta, \EXP_{(x, 0)}(g_\theta(X_t)Y_te^{\theta Y_t})\rangle +  \langle \psi_\theta, \cL(\theta, t)g_\theta'\rangle
	& = t\mu'(\theta)e^{t\mu(\theta)}\langle \psi_\theta, g_\theta \rangle +  e^{t\mu(\theta)}\langle \psi_\theta,g'_\theta\rangle,
	\end{align*}
	which simplifies to 
	\begin{align*}
	\langle \psi_\theta, \EXP_{(x, 0)}(g_\theta(X_t)Y_te^{\theta Y_t})\rangle +  e^{t\mu(\theta)}\langle \psi_\theta, g_\theta'\rangle
	& = t\mu'(\theta)e^{t\mu(\theta)} +  e^{t\mu(\theta)}\langle \psi_\theta,g'_\theta\rangle.
	\end{align*}
	Thus, we obtain the following formula for $\mu'(\theta)$.
	\begin{align}
	\mu'(\theta) = \frac{\langle \psi_\theta, \EXP_{(x, 0)}(g_\theta(X_t)Y_te^{\theta Y_t})\rangle}{te^{t\mu(\theta)}}.
	\end{align}
	Differentiating the equation \eqref{1stder} again with respect to $\theta$
	and taking the action of the linear functional $\psi_\theta$ on both sides, we obtain, 
	\begin{align*}
	&\langle \psi_\theta, \EXP_{(x, 0)}(g_\theta(X_t)Y_t^2e^{\theta Y_t})\rangle+ 2\langle \psi_\theta, \EXP_{(x, 0)}(g'_\theta(X_t)Y_te^{\theta Y_t})\rangle  + e^{t\mu(\theta)}\langle \psi_\theta, g_\theta''\rangle\\
	&= t\mu''(\theta)e^{t\mu(\theta)} + t^2 (\mu'(\theta))^2 e^{t\mu(\theta)} + 2 t\mu'(\theta) e^{t\mu(\theta)} \langle \psi_\theta, g_\theta'\rangle + e^{t\mu(\theta)} \langle \psi_\theta, g_\theta''\rangle. 
	\end{align*}
	Thus, rearranging the terms, we obtain the following formula for $\mu''(\theta)$:
	\begin{align*}
	\mu''(\theta) &= \frac{\langle \psi_\theta, \EXP_{(x, 0)}(g_\theta(X_t)Y_t^2e^{\theta Y_t})\rangle - t^2 (\mu'(\theta))^2 e^{t\mu(\theta)}}{te^{t\mu(\theta)}}\\& \phantom{aaaaaaaaaaaa}+ 2 \frac{\langle \psi_\theta, \EXP_{(x, 0)}(g'_\theta(X_t)Y_te^{\theta Y_t})\rangle -t\mu'(\theta) e^{t\mu(\theta)}\langle \psi_\theta, g_\theta'\rangle}{te^{t\mu(\theta)}}.
	\end{align*}
	Using the formula for $\mu'(\theta)$ in the above expression we obtain
	\begin{align}\label{2ndder}
	&\mu''(\theta) = \frac{\langle \psi_\theta, \EXP_{(x, 0)}(g_\theta(X_t)Y_t^2e^{\theta Y_t - t\mu(\theta)})\rangle - (\langle \psi_\theta, \EXP_{(x, 0)}(g_\theta(X_t)Y_te^{\theta Y_t - t\mu(\theta)})\rangle)^2 }{t} \\
	&+ 2 \frac{\langle \psi_\theta, \EXP_{(x, 0)}(g'_\theta(X_t)Y_te^{\theta Y_t - t\mu(\theta)})\rangle -\langle \psi_\theta, \EXP_{(x, 0)}(g_\theta(X_t)Y_te^{\theta Y_t- t\mu(\theta)})\rangle\langle \psi_\theta, g_\theta'\rangle}{t}. \nonumber 
	\end{align}
	Let $\tilde{\Ban}$ be the Banach space  of bounded continuous functions defined on $M \times \real$ equipped with the supremum norm.  We define a new family of bounded linear operators $N(\theta, t): \tilde\Ban \to \tilde\Ban$, $t \geq 0$ by
	\begin{equation}\label{newop}
	N(\theta, t)f(x_0, y_0) := \EXP_{(x_0,y_0)}\Big(f(X_t, Y_t)e^{\theta( Y_t - y_0) - t\mu(\theta)}\frac{g_\theta(X_t)}{g_\theta(x_0)}\Big)
	\end{equation}
	for each $f \in \tilde{\Ban}$. Note that the family $\{N(\theta, t)\}_{t \geq 0}$ forms a $C^0$ semigroup.
	
	We first observe that the operators $\{N(\theta, t)\}_{t \geq 0}$ are positive, and $N(\theta, t){\bf{1}} = {\bf{1}}$, where~${\bf{1}}$ denotes the constant function taking the value $1$ on $M \times \real$.
	
	The operator $N(\theta, t)$ is also an operator on $\Ban$ because, for $f \in \Ban$,
	\[N(\theta, t)f(x_0) = \EXP_{(x_0,y_0)}\Big(f(X_t)e^{\theta( Y_t - y_0) - t\mu(\theta)}\frac{g_\theta(X_t)}{g_\theta(x_0)}\Big) \]\[ = \Big[\frac{e^{-t\mu(\theta)}}{g_\theta}\cL(\theta, t)(g_\theta f)\Big](x_0) \in \Ban.\]
	Now, corresponding to this family of operators, we have a new Markov process $(\tilde X_t, \tilde Y_t)$on $M \times \real$, such that,  $N(\theta, t)f(x_0, y_0) = \EXP_{(x_0,y_0)}(f(\tilde X_t, \tilde Y_t))$.
	In addition, we observe that $\langle \psi_\theta g_\theta , N(\theta, t)f \rangle = \langle \psi_\theta g_\theta , f \rangle $ for all $f \in \Ban$. That is, $\psi_\theta g_\theta$ is the invariant measure for the process $\wt{X}_t$ on the manifold $M$ for all $t \geq 0$.
	
	Let us define the function $h \in \tilde{\Ban}$ by $h(x,y) = y$ for all $(x,y) \in M\times \real$. Now, we re-write the formula \eqref{2ndder} for $\mu''(\theta)$ as 
	\begin{align*}
	&\mu''(\theta) =\frac{1}{t}\Big(\langle \psi_\theta(x), N(\theta,t)(h^2)(x,0)g_\theta(x)\rangle_x - (\langle \psi_\theta(x), N(\theta,t)(h)(x,0)g_\theta(x)\rangle_x)^2\Big)\\ & +\frac{2}{t} \Big[\langle \psi_\theta(x), N(\theta, t)\Big(\frac{hg_\theta'}{g_\theta}\Big)(x,0)g_\theta(x)\rangle_x \\ &-\langle \psi_\theta(x), N(\theta,t)(h)(x,0)g_\theta(x)\rangle_x\langle \psi_\theta(x), g_\theta'(x)\rangle_x\Big]\\
	&=\frac{1}{t}\Big(\langle \psi_\theta g_\theta, N(\theta,t)(h^2)\rangle - (\langle \psi_\theta g_\theta, N(\theta,t)(h)\rangle)^2\Big)\\ &+ \frac{2}{t} \Big(\langle \psi_\theta g_\theta, N(\theta, t)\Big(\frac{hg_\theta'}{g_\theta}\Big)\rangle-\langle \psi_\theta g_\theta, N(\theta,t)(h)\rangle\langle \psi_\theta g_\theta, \frac{g_\theta'}{g_\theta}\rangle\Big).
	\end{align*}
	Therefore, we have, 
	\begin{align*}
	\mu''(\theta)
	&= \frac{1}{t}\Big(\langle \psi_\theta g_\theta, \EXP_{(x,0)}(\wt{Y}_t^2)\rangle - (\langle \psi_\theta g_\theta, \EXP_{(x,0)}(\wt{Y}_t)\rangle)^2\Big)\\ &\phantom{aaaaaaa}+ \frac{2}{t} \Big(\langle \psi_\theta g_\theta, \EXP_{(x,0)}\Big(\frac{\wt{Y}_t g_\theta'(\wt{X}_t)}{g_\theta(\wt{X}_t)}\Big)\rangle-\langle \psi_\theta g_\theta,\EXP_{(x,0)}(\wt{Y}_t)\rangle\langle \psi_\theta g_\theta, \frac{g_\theta'}{g_\theta}\rangle\Big).
	\end{align*}
	Denoting $\langle \psi_\theta g_\theta, \EXP_{(x,0)}(f(\wt{X}_t, \wt{Y}_t))\rangle$ by $\EXP_{ \psi_\theta g_\theta}(f(\wt{X}_t, \wt{Y}_t))$, the above formula can be written as 
	\[
	\mu''(\theta)
	= \frac{1}{t}\Big( \EXP_{\psi_\theta g_\theta}(\wt{Y}_t^2) -(\EXP_{\psi_\theta g_\theta}(\wt{Y}_t))^2\Big) +\]	\begin{equation}\label{2ndder2}
	+ \frac{2}{t} \Big(\EXP_{\psi_\theta g_\theta}\Big(\frac{\wt{Y}_t g_\theta'(\wt{X}_t)}{g_\theta(\wt{X}_t)}\Big)- \EXP_{\psi_\theta g_\theta}(\wt{Y}_t)\langle \psi_\theta g_\theta, \frac{g_\theta'}{g_\theta}\rangle\Big).	\end{equation}
	
Now, in order to prove that $\mu''(\theta)>0$, we first show that the first term in \eqref{2ndder2} is the effective diffusivity of the process $\tilde Y_t$, which is strictly positive. Then we prove that that the second term in \eqref{2ndder2} goes to zero as $t$ goes to infinity, since the processes $\tilde X_t$ and $\tilde Y_t$ de-correlate as as $t$ goes to infinity.
	
	In order to analyze the process $(\wt{X}_t, \wt{Y}_t)$, we first study the transition kernel of the associated Markov Operator $N(\theta,t)$. 
	For $f \in \tilde{\Ban}$, 
	\[
	N(\theta, t)f(x_0, y_0) = \int_M\int_{\real} f(x,y)k(t, (x_0, y_0), (x,y))\, dy \, dx,
	\]
	where $$k(t, (x_0, y_0), (x,y)) := e^{-t\mu(\theta)}\frac{e^{\theta y}g_\theta(x)}{e^{\theta y_0}g_\theta(x_0)} p(t, (x_0, y_0), (x,y)).$$
	From \eqref{transden}, we see that $k(t, (x_0, y_0), (x,y))$ solves the PDE
	\[
	\partial_t k = g_\theta(x)e^{\theta y}\cM^*_{(x,y)}\Big(\frac{k}{g_\theta(x)e^{\theta y}}\Big) - \mu(\theta)k  =: \wt{\cM}^* k,
	\]
	\[
	k(0, (x_0, y_0), (x,y)) = \delta_{(x_0, y_0)}(x,y),
	\]
	where, we have a new differential operator $\wt{\cM}$ acting on functions $u:~M\times\real \to\real$ given by 
	\[
	\wt{\cM}^*u = g_\theta(x)e^{\theta y}\cM^*_{(x,y)}\Big(\frac{u}{g_\theta(x)e^{\theta y}}\Big) - \mu(\theta)u.
	\] 
	Observe that 
	\begin{align*}
	\wt{\cM}^*k
	&= \cM^*k - \frac{\nabla_xg_\theta}{g_\theta}(V(x)V^T(x))\nabla_x k -  \theta \sigma^2(x)\nabla_y k  \\&+ \Big[\frac{V_0(x)\nabla_x g_\theta(x)}{g_\theta(x)}   +\frac{1}{2}\theta^2 \sigma^2(x) -\frac{\nabla_x((V(x)V^T(x))\nabla_xg_\theta(x))}{2g_\theta(x)}\\ & \frac{(\nabla_xg_\theta)^2}{g_\theta^2}(V(x)V^T(x))  + b(x)\theta -   \mu(\theta) \Big]k.
	\end{align*}
	From the choice of $g_\theta$, we know that $e^{-\theta y}\cM(e^{\theta y}g_\theta(x)) = \mu(\theta)g_\theta(x)$. That is,
	\[
	\frac{1}{2} \nabla_x [(V(x)V^T(x))\nabla_x g_\theta] + V_\theta \nabla_x g_\theta + b(x) \theta g_\theta  +  \frac{1}{2} (\sigma^2(x))\theta^2 g_\theta= \mu(\theta) g_\theta.
	\]
	Therefore, the above expression simplifies to 
	\begin{align*}
	\wt{\cM}^*k
	&= \cM^*k - \frac{\nabla_xg_\theta}{g_\theta}(V(x)V^T(x))\nabla_x k -  \theta \sigma^2(x)\nabla_y k \\ &+\Big(\frac{(\nabla_xg_\theta)^2(V(x)V^T(x))}{g_\theta^2} - \frac{\nabla_x [(V(x)V^T(x))\nabla_x g_\theta]}{g_\theta} \Big)k.
	\end{align*}
	Thus, the operator $\wt{\cM}$ simplifies to 
	\[
	\wt{\cM}k = \cM k + \frac{\nabla_xg_\theta}{g_\theta}(V(x)V^T(x))\nabla_x k +  \theta \sigma^2(x)\nabla_y k.
	\]
	From the above expression of the generator of the new process $(\wt{X}_t, \wt{Y}_t)$, we conclude that the process $(\wt{X}_t, \wt{Y}_t)$ differ from the process $(X_t, Y_t)$ only by the additional drift terms in $x$ and $y$.
	 The  Effective Diffusivity of the process $\wt{Y}_t$ is given by 
	 \[
	 \Xi := \lim_{t \to \infty} \frac{\EXP_{\psi_\theta g_\theta}\Big(\big(\wt{Y}_t-\EXP_{\psi_\theta g_\theta}\tilde{Y}_t\big)^2\Big)}{t}. 
	 \] 

We now show that the effective diffusivity of the process $\wt{Y}_t$ is also strictly positive. Let $c_{\theta} \in \real$ be given by,
		\[
		c_{\theta} = \int_M(b+\theta\sigma^2)\psi_{\theta}g_{\theta}.
		\]
		Choose a function $f : M\to \real$ such that $\wt{\cM}f + b + \sigma^2\theta = c_\theta$ on $M$. The existence of such a function $f$ is guaranteed because $\int_M (b + \theta\sigma^2 - c_\theta)\psi_\theta g_\theta = 0$. The process $\wt{Y}_t + f(\wt{X}_t) - c_\theta t$ forms a martingale, and therefore,
		\begin{align*}
		\wt{Y}_t + f(\wt{X}_t) - c_\theta t - \wt{Y}_0 - f(\wt{X}_0)&= \int_0^tV(\wt{X}_u)\nabla_xf(\wt{X}_u)\, dW_u + \int_0^t\sigma(\wt{X}_u)\, d\wt{W}_u + \\&+\int_0^t (\wt{\cM}f(\wt{X}_u) + b(\wt{X}_u) + \theta \sigma^2(\wt{X}_u) -c_\theta )\,du \nonumber \\
		&= \int_0^tV(\wt{X}_u)\nabla_xf(\wt{X}_u)\, dW_u + \int_0^t\sigma(\wt{X}_u)\, d\wt{W}_u.\nonumber
		\end{align*}
		Thus,
		\begin{align*}
		&\EXP_{\psi_\theta g_\theta}\big(\wt{Y}_t-\EXP_{\psi_\theta g_\theta}\wt{Y}_t\big)^2 \\&=  \EXP_{\psi_\theta g_\theta}\Big(\int_0^tV(\wt{X}_u)\nabla_xf(\wt{X}_u)\, dW_u + \int_0^t\sigma(\wt{X}_u)\, d\wt{W}_u - \big(f(\wt{X}_t) - \EXP{\psi_\theta g_\theta} (f(\wt{X}_t))\big)\Big)^2\\
		& = \EXP_{\psi_\theta g_\theta}\Big(\frac{1}{2}\int_0^t(V(\wt{X}_u)\nabla_xf(\wt{X}_u))(V(\wt{X}_u)\nabla_xf(\wt{X}_u))^*\,du \Big)  \\
		&~~~+\EXP_{\psi_\theta g_\theta}\Big(\frac{1}{2}\int_0^t\sigma^2(\wt{X}_u)\,du \Big)+\EXP_{\psi_\theta g_\theta}(f(\wt{X}_t)^2) -  \EXP_{\psi_\theta g_\theta}(f(\wt{X}_t))^2\\&~~~-2\EXP_{\psi_\theta g_\theta}(f(\wt{X}_t))\Big(\int_0^tV(\wt{X}_u)\nabla_xf(\wt{X}_u)\, dW_u + \int_0^t\sigma(\wt{X}_u)\, d\wt{W}_u \Big)
		\end{align*}
		Also, note that 
		\[
		\lim_{t \to \infty}\frac{\EXP_{\psi_\theta g_\theta}(f(\wt{X}_t))\Big(\int_0^tV(\wt{X}_u)\nabla_xf(\wt{X}_u)\, dW_u + \int_0^t\sigma(\wt{X}_u)\, d\wt{W}_u \Big)}{t} = 0.
		\] 
		Therefore, using the fact that $\psi_\theta g_\theta $ is the invariant measure of the process $\wt{X}_t$ on $M$, we have, 
		\begin{align*}
		\Xi& = \frac{1}{2}\int_M\big((V(x)\nabla_xf(x))(V(x)\nabla_xf(x))^* + \sigma^2(x)\big)\psi_\theta g_\theta \,dx  \\ &\phantom{\frac{1}{2}\int_M\big((V(x)\nabla_xf(x))(V(x)\nabla_xf(x))^* }\\&+ \lim_{t \to \infty} \frac{\EXP_{\psi_\theta g_\theta}(f(\wt{X}_t)^2) -  \EXP_{\psi_\theta g_\theta}(f(\wt{X}_t))^2}{t}. 
		\end{align*}
		Since $\sigma>0$ for all $x \in M$, we have $\Xi >0$.\qed\
\\
\\
 Thus we have shown that, the first term in \eqref{2ndder2} is positive. Now it remains to show that the limit of the second term in $\eqref{2ndder2}$ is zero as $t$ approaches infinity. In other words, the processes $\tilde X_t$ and $\tilde Y_t$ de-correlate as as $t$ goes to infinity. Thus, we need to show
	\[
	\lim_{t \to \infty} \frac{\EXP_{\psi_\theta g_\theta}\big(\frac{\wt{Y}_t g_\theta'(\wt{X}_t)}{g_\theta(\wt{X}_t)}\big)- \EXP_{\psi_\theta g_\theta}(\wt{Y}_t)\langle \psi_\theta g_\theta,\frac{ g_\theta'}{g_\theta}\rangle }{t}  = 0.
	\]
	First, we observe that 
	\begin{equation}\label{meanzero}
	\lim_{t \to \infty} \frac{1}{t} \EXP_{\psi_\theta g_\theta}(\wt{Y}_t) - c_\theta =	\end{equation}\[=\lim_{t \to \infty} \frac{1}{t} \EXP_{\psi_\theta g_\theta}  \Big(\wt{Y}_0+ f(\wt{X}_0 ) - f(\wt{X}_t)+ \int_0^tV(\wt{X}_u)\nabla_xf(\wt{X}_u)\, dW_u + \int_0^t\sigma(\wt{X}_u)\, d\wt{W}_u\Big) = 0.\]
	Therefore,
	\[
	\lim_{t \to \infty} \frac{ \EXP_{\psi_\theta g_\theta}(\wt{Y}_t)\langle \psi_\theta g_\theta,\frac{ g_\theta'}{g_\theta}\rangle}{t} = c_\theta\langle \psi_\theta g_\theta,\frac{ g_\theta'}{g_\theta}\rangle.
	\]
	Thus, we only need to show that \[
	\lim_{t \to \infty} \frac{1}{t}\EXP_{\psi_\theta g_\theta}\bigg(\frac{\wt{Y}_t g_\theta'(\wt{X}_t)}{g_\theta(\wt{X}_t)}\bigg)  = c_\theta\langle \psi_\theta g_\theta,\frac{ g_\theta'}{g_\theta}\rangle,
	\]
	that is, to show that 
	\[
	\lim_{t \to \infty} \frac{1}{t}\EXP_{\psi_\theta g_\theta}\bigg(\frac{(\wt{Y}_t - c_\theta t) g_\theta'(\wt{X}_t)}{g_\theta(\wt{X}_t)}\bigg) = 0
	\] 		
	Since $0 < \Xi < \infty$, there exists a constant $K >0$ such that 
	\begin{equation*}\label{l2bnd}
	\EXP_{\psi_\theta g_\theta} \frac{ (\wt{Y}_t - c_\theta t)^2}{t} \leq K
	\end{equation*}
	Using Cauchy- Schwartz inequality, and the upper bound on $\EXP_{\psi_\theta g_\theta}((\wt Y_t - c_\theta t)^2)$, stated above, we have 
	\begin{align*}
	\EXP_{\psi_\theta g_\theta}\bigg(\Big|\frac{(\wt{Y}_t - c_\theta t) g_\theta'(\wt{X}_t)}{g_\theta(\wt{X}_t)}\Big|\bigg) &\leq \EXP_{\psi_\theta g_\theta}\big((\wt{Y}_t - c_\theta t)^2\big)^{1/2} \EXP_{\psi_\theta g_\theta}\bigg(\frac{(g_\theta'(\wt{X}_t))^2}{g^2_\theta(\wt{X}_t)}\bigg)^{1/2} \\ &\leq \sqrt{K}\sqrt{t} \sup_{x \in M} \Big| \frac{(g_\theta'(x))^2}{g^2_\theta(x)}\Big|
	\end{align*}
	Therefore, we have, 
	\[
	\lim_{t \to \infty} \frac{1}{t}\EXP_{\psi_\theta g_\theta}\bigg(\Big|\frac{(\wt{Y}_t - c_\theta t) g_\theta'(\wt{X}_t)}{g_\theta(\wt{X}_t)}\Big|\bigg)  \leq \lim_{t \to \infty} \frac{1}{t} \sqrt{K}\sqrt{t} \sup_{x \in M} \Big| \frac{(g_\theta'(x))^2}{g^2_\theta(x)}\Big|= 0.
	\]
	We have shown that the conditions (D1), (D2) and (D3) hold with $r_1$ arbitrarily large. As a result, for all $r$, $Y_t$ admits the weak and strong expansion for LDP of order $r$ in the range $(0,\infty)$.
		
\end{proof}	
	
\subsection*{Acknowledgement} 
The authors would like to thank Dmitry Dolgopyat and Leonid Koralov for useful discussions and suggestions during the project and carefully reading the manuscript. While working on this article, P. Hebbar was partially supported by the ARO grant W911NF1710419.

	\renewcommand{\baselinestretch}{1}
	\small\normalsize
	\bibliographystyle{unsrt}
	\bibliography{bibpratima}
	\end{document}